\documentclass[11pt]{amsart}

\usepackage[margin=1in,a4paper]{geometry}

\usepackage{setspace}
\usepackage{mathtools}

\usepackage{upgreek}

\usepackage{tikz-cd}

\usepackage[shortlabels]{enumitem}

\usepackage[implicit=false]{hyperref}

\usepackage[initials,lite]{amsrefs}

\newtheorem{theorem}{Theorem}
\newtheorem{proposition}{Proposition}
\newtheorem{corollary}{Corollary}
\newtheorem{lemma}{Lemma}

\theoremstyle{definition}

\newtheorem{definition}{Definition}

\newcommand{\RR}{\mathbb{R}}

\newcommand{\Linf}{$L_\infty$}
\newcommand{\Ainf}{$A_\infty$}

\newcommand{\F}{\mathbb{F}}

\DeclareMathOperator{\MC}{MC}
\DeclareMathOperator{\End}{End}

\newcommand{\h}{\mathbf{h}}
\newcommand{\p}{\mathbf{p}}
\renewcommand{\i}{\mathbf{i}}

\newcommand{\Gg}{\mathfrak{g}}

\DeclareMathOperator{\C}{\mathsf{C}}

\DeclareMathOperator{\Hom}{Hom}

\newcommand{\LL}{\widetilde{\mathsf{Lie}}}
\renewcommand{\L}{\mathsf{Lie}}
\renewcommand{\S}{\mathsf{Kan}}

\newcommand{\CF}{\mathcal{F}}
\newcommand{\CV}{\mathcal{V}}

\newcommand{\CW}{\mathcal{W}}
\DeclareMathOperator{\Ho}{Ho}
\DeclareMathOperator{\gr}{gr}

\newcommand{\D}{\mathsf{D}}

\newcommand{\G}{\mathcal{G}}

\newcommand{\ohat}{\widehat{\otimes}}

\newcommand{\DR}{\Upomega}

\newcommand{\pullback}%
{\arrow[dr,phantom,"\scalebox{1.5}{$\lrcorner$}",very near start]}

\renewcommand{\d}{\mathsf{d}}

\let\exp\undefined

\DeclareMathOperator{\exp}{\mathsf{e}}

\begin{document}

\title{Higher holonomy for curved \Linf-algebras 1: simplicial methods.}

\author{Ezra Getzler} \address{Department of Mathematics, Northwestern
  University, Evanston, Illinois, USA}

\maketitle

\section*{Introduction}

In this article, we work with pro-nilpotent curved
\Linf-algebras. These generalize nilpotent differential graded (dg)
Lie algebras.

A \textbf{filtered graded vector space} is a graded vector space over
a field $\F$ with a complete decreasing filtration
\begin{equation*}
  V = F^0V \supset F^1V \supset \ldots \supset 0
\end{equation*}
of graded subspaces. Filtered graded vector spaces form a symmetric
monoidal category: the completed tensor product is
\begin{equation*}
  F^p(V \ohat W) = \lim_{q\to\infty} \biggl( \sum_{i+j=p} F^iV \otimes F^jW \biggr)
  \Bigm/ \biggl( \sum_{i+j=q} F^iV \otimes F^jW \biggr) .
\end{equation*}
We assume that $\F$ has characteristic zero.

A \textbf{curved \Linf-algebra} is a filtered graded vector space $L$
with multilinear brackets
\begin{equation*}
  \{x_1,\ldots,x_k\} : F^{p_1}L^{\ell_1} \times \ldots \times F^{p_k}L^{\ell_k} \to
  F^{p_1+\ldots+p_k}L^{\ell_1+\cdots+\ell_k+1} , \quad k\ge0 ,
\end{equation*}
satisfying the following conditions:
\begin{enumerate}[1)]
\item for $1\le i<k$,
  $\{x_1,\ldots,x_i,x_{i+1},\ldots,x_k\} = (-1)^{\ell_i\ell_{i+1}}
  \{x_1,\ldots,x_{i+1},x_i,\ldots,x_k\}$;
\item for $n\ge0$,
  \begin{equation*}
    \hspace{-\leftmargin}
    \sum_{\pi\in S_n} \sum_{k=0}^n \frac{(-1)^\epsilon}{k!(n-k)!} \,
    \{\{x_{\pi(1)},\ldots,x_{\pi(k)}\},x_{\pi(k+1)},\ldots,x_{\pi(n)}\} = 0 ,
  \end{equation*}
  where $(-1)^\epsilon$ is the sign associated by the Koszul sign rule to the
  action of the permutation $\pi$ on the elements $(x_1,\ldots,x_n)$ of the
  graded vector space $L$.
\end{enumerate}
An \Linf-algebra is a curved \Linf-algebra such that
$\{x_1,\ldots,x_k\}=0$ for $k=0$. The \textbf{curvature} of a curved
\Linf-algebra $L$ is the element $\{\}\in L^1$.

A curved \Linf-algebra is \textbf{pro-nilpotent} if $L=F^1L$. In this
article, all curved \Linf-algebras are assumed to be
pro-nilpotent.

\begin{definition}
  Let $L$ be a curved \Linf-algebra. Its \textbf{Maurer--Cartan locus}
  $\MC(L)$ is the set of solutions of the Maurer--Cartan equation
  \begin{equation*}
    \MC(L) = \biggl\{ x \in L^0 \biggm| \sum_{n=0}^\infty \frac{1}{n!} \{x^{\otimes
        n}\} = 0 \in L^1 \biggr\} .
  \end{equation*}
\end{definition}

The Maurer-Cartan equation makes sense because $L$ is pro-nilpotent:
the $n$-bracket $\{x^{\otimes n}\}$ is in $F^nL$, and the filtered graded
vector space $L$ is complete.

Differential graded (dg) Lie algebras are the special case of
\Linf-algebras in which all brackets vanish except the linear bracket
$\{x_1\}$ and the bilinear bracket $\{x_1,x_2\}$. To recover the usual
definition of a dg Lie algebra, replace $L$ by $\Gg=L[-1]$, with
differential $\delta=(-1)^{\ell_1} \{x_1\}$ (which is a differential since the
curvature vanishes) and Lie bracket
$[x_1,x_2]=(-1)^{\ell_1}\{x_1,x_2\}$. We hope that this shift in grading
conventions between dg Lie algebras and curved \Linf-algebras does not
lead to confusion.

If $\DR$ is a dg commutative algebra and $L$ is a curved
\Linf-algebra, the completed graded tensor product $\DR\ohat L$ is a
curved \Linf-algebra, with brackets
\begin{equation*}
  \{\alpha_1\otimes x_1,\ldots,\alpha_n\otimes x_n\} =
  \begin{cases}
    1 \otimes \{\} , & n=0 , \\[3pt]
    d\alpha_1 \otimes x_1 + (-1)^{|\alpha_1|} \, \alpha_1 \otimes \{x_1\} , & n=1 , \\[3pt]
    (-1)^{\sum_{i<j}|x_i||\alpha_j|} \, \alpha_1\ldots\alpha_n \otimes \{x_1,\ldots,x_n\} , & n>1 ,
  \end{cases}
\end{equation*}
and filtration $F^p(\DR\ohat L) = \DR \ohat F^pL$.

Let $\DR_n$ be the dg commutative algebra
\begin{equation*}
  \DR_n = \F[t_0,\ldots,t_n,dt_0,\ldots,dt_n]/(t_0+\cdots+t_n-1,dt_0+\cdots+dt_n) .
\end{equation*}
As $n$ varies, we obtain a simplicial dg commutative algebra
$\DR_\bullet$. If $\F=\RR$ is the field of real numbers, we may identify
$\DR_n$ with the algebra of polynomial coefficient differential forms
on the convex hull $|\triangle^n|$ of the $n+1$ basis vectors
$\{e_i\mid 0\le i\le n\}$ of $\RR^{0,\ldots,n}$.

\begin{definition}
  The \textbf{nerve} $\MC_\bullet(L)$ of a curved \Linf-algebra is the
  Maurer--Cartan locus of the completed tensor product $\DR_\bullet\ohat L$:
  \begin{equation*}
    \MC_\bullet(L) = \MC(\DR_\bullet\ohat L) .
  \end{equation*}
\end{definition}

The nerve was introduced by Hinich \cite{Hinich}.  In \cite{Linf}, we
show that $\MC_\bullet(L)$ is a Kan complex when $L$ is a nilpotent and has
vanishing curvature, but the proof extends to the current setting
without modification.

Let $h:L\to L[-1]$ be a map of degree $-1$ on the underlying filtered
graded vector space of the curved \Linf-algebra $L$. Consider the
sublocus of the Maurer--Cartan locus satisfying the gauge condition
$hx=0$:
\begin{equation*}
  \MC(L,h) = \{ x\in\MC(L) \mid hx = 0 \} .
\end{equation*}
As in \cite{Linf}, we only consider gauges $h$ that define a
contraction.

The condition $hx=0$ is analogous to the Lorenz gauge
$\mathop{\mathrm{div}} A=0$ in Maxwell's theory of electromagnetism,
where $A$ is a connection 1-form on a complex line bundle. This gauge
is used by Kuranishi \cite{Kuranishi} to study the Kodaira--Spencer
equation (the Maurer--Cartan equation for the Dolbeault resolution
$A^{0,\ast}(X,T)$ of the sheaf of Lie algebras of holomorphic vector
fields on a complex manifold $X$).

In \cite{Linf}, we introduced the gauge condition corresponding to
Dupont's homotopy $s_\bullet$ on $\DR_\bullet$. We now recall the definition of
$s_\bullet$.

The vector field
\begin{equation*}
  E_i = \sum_{j=0}^n (t_j-\delta_{ij}) \partial_j
\end{equation*}
on $|\triangle^n|$ generates the dilation flow $\phi_i(u)$ centered at the
$i$th vertex of $|\triangle^n|$. Let $\epsilon^i_n:\DR_n\to\F$ be evaluation at
$e_i$. The Poincar\'e homotopy
\begin{equation*}
  h^i_n = \int_0^1 \phi_i(u) \, \iota(E_i) \, \frac{du}{u}
\end{equation*}
is a chain homotopy between the identity and $\epsilon^i_n$:
\begin{equation*}
  dh^i_n + h^i_nd = 1 - \epsilon^i_n .
\end{equation*}

Whitney's complex of elementary differential forms is the subcomplex
$W_n\subset\DR_n$ with basis
\begin{equation*}
  \omega_{i_0\ldots i_k} = k! \, \sum_{j=0}^k (-1)^i t_{i_j} dt_{i_0} \ldots
  \widehat{dt}{}_{i_j} \ldots dt_{i_k} , \quad 0\le i_0<\ldots<i_k\le n .
\end{equation*}
It is naturally isomorphic to the complex $N^*(\triangle^n,\F)$ of normalized
simplicial cochains on the $n$-simplex. The operator
\begin{equation*}
  p_n = \sum_{k=0}^n (-1)^k \sum_{i_0<\ldots<i_k} \omega_{i_0\ldots i_k} \epsilon_n^{i_k}
  h_n^{i_{k-1}} \ldots h_n^{i_0}
\end{equation*}
is a projection $p_n$ onto the subcomplex $W_n\subset\DR_n$.

Dupont \cite{Dupont} constructs a simplicial homotopy
\begin{equation*}
  s_n = \sum_{k=0}^{n-1} \sum_{0\le i_0<\ldots<i_k\le n} \omega_{i_0\ldots i_k} h_n^{i_k} \ldots h_n^{i_0}
\end{equation*}
satisfying
\begin{equation*}
  ds_n + s_n d = 1 - p_n .
\end{equation*}

\begin{definition}
  The simplicial subcomplex $\gamma_\bullet(L)\subset\MC_\bullet(L)$ is the simplicial subset
  of Maurer--Cartan elements on which $s_\bullet$ vanishes:
  \begin{equation*}
    \gamma_\bullet(L)  = \MC\bigl( \DR_\bullet\ohat L,s_\bullet \bigr) .
  \end{equation*}
\end{definition}

We now describe the functoriality of $\MC_\bullet(L)$ and $\gamma_\bullet(L)$.
\begin{definition}
  A \textbf{morphism} $f:L\to M$ of curved \Linf-algebras is a sequence
  of filtered graded symmetric maps
  \begin{equation*}
    f = f_{(k)} : F^{p_1}L^{\ell_1} \times \ldots \times F^{p_k}L^{\ell_k} \to
    F^{p_1+\ldots+p_k}M^{\ell_1+\ldots+\ell_k} , \quad k\ge0 ,
  \end{equation*}
  such that for all $n\ge0$,
  \begin{multline*}
    \sum_{\pi\in S_n} \sum_{k=0}^\infty \frac{(-1)^\epsilon}{k!} \sum_{n_1+\cdots+n_k=n}
    \frac{1}{n_1!\ldots n_k!} \,
    \{f_{(n_1)}(x_{\pi(1)},\ldots),\ldots,f_{(n_k)}(\ldots,x_{\pi(n)})\} \\
    = \sum_{\pi\in S_n} \sum_{k=0}^n \frac{(-1)^\epsilon}{k!(n-k)!} \,
    f(\{x_{\pi(1)},\ldots,x_{\pi(k)}\},x_{\pi(k+1)},\ldots,x_{\pi(n)}) .
  \end{multline*}

  The composition $g\bullet f$ of morphisms $f:L\to M$ and $g:M\to N$ is
  \begin{multline*}
    (g\bullet f)(x_1,\ldots,x_n) = \sum_{\pi\in S_n} \sum_{k=0}^\infty
    \frac{(-1)^\epsilon}{k!} \sum_{n_1+\cdots n_k} \frac{1}{n_1!\ldots n_k} \\
    g_{(k)}\bigl( f_{(n_1)}(x_{\pi(1)},\ldots),\ldots,f_{(n_k)}(\ldots,x_{\pi(n_k)}) \bigr) .
  \end{multline*}
\end{definition}

A morphism $f:L\to M$ is \textbf{strict} if $f_{(k)}=0$, $k\ne1$. Curved
\Linf-algebras form a category $\LL$; denote the subcategory of strict
morphisms by $\L$.

The set of points of an object $X$ in a category is the set of
morphisms from the terminal object of the category to $X$. The
terminal object in the category $\LL$ is the curved \Linf-algebra $0$,
and the set of points $\Hom(0,L)$ of a curved \Linf-algebra $L$ is the
Maurer-Cartan set $\MC(L)$. This shows that $\MC(L)$ is a left-exact
functor from the category $\LL$ of curved \Linf-algebras to the
category of sets. The action of a morphism $f:L\to M$ on a Maurer-Cartan
element $x\in\MC(L)$ is given by the formula
\begin{equation*}
  f(x) = \sum_{k=0}^\infty \frac{1}{k!} \, f_{(k)}(x,\ldots,x) .
\end{equation*}

In this article, following \cite{Linf}, we work with $\gamma_\bullet$ as a
functor on the category $\L$ of \Linf-algebras with strict morphisms.
Robert-Nicoud and Vallette \cite{RV} have shown that $\gamma_\bullet$ extends to
a larger category $\LL_\pi$ with the same objects as $\L$. The inclusion
$\L\subset\LL$ factors through the inclusion $\L\subset\LL_\pi$, though the natural
functor from $\LL_\pi$ to $\LL$ is neither faithful nor full.

The space $\C(L)$ of Chevalley-Eilenberg chains of a curved
\Linf-algebra $L$ is the filtered coalgebra
\begin{equation*}
  \C(L) = \prod_{k=0}^\infty \bigl( L^{\ohat k} \bigr)_{S_k} .
\end{equation*}
(Taking the product over $k$ instead of the sum is equivalent to
taking the completion, by the hypothesis that $L$ is pro-nilpotent.)
It is a filtered dg cocommutative coalgebra, with coproduct
\begin{equation*}
  \nabla \bigl( x_1\otimes\ldots\otimes x_k \bigr) = \sum_{\pi\in S_k}
  \sum_{j=0}^k \frac{(-1)^\epsilon}{j!(k-j)!} \, \bigl( x_{\pi(1)} \otimes\ldots\otimes x_{\pi(j)}
  \bigr) \otimes \bigl( x_{\pi(j+1)} \otimes \cdots \otimes x_{\pi(k)} \bigr)
\end{equation*}
and differential
\begin{equation*}
  \delta(x_1\otimes\ldots\otimes x_k) = \sum_{\pi\in S_k} \sum_{j=0}^k \frac{(-1)^\epsilon}{j!(k-j)!} \,
  \{x_{\pi(1)},\ldots,x_{\pi(j)}\} \otimes x_{\pi(j+1)}\otimes\cdots\otimes x_{\pi(k)} .
\end{equation*}
The coproduct
\begin{equation*}
  \nabla: F^p\C(L) \to \bigoplus_{q=0}^pF^q\C(L) \ohat F^{p-q}\C(L)
\end{equation*}
and codifferential $\delta:F^p\C(L)\to F^p\C(L)$ have filtration degree $0$.

A morphism $f:L\to M$ of \Linf-algebras induces a morphism of filtered
dg cocommutative coalgebras $\C(f):\C(L)\to \C(M)$, by the formula
\begin{equation*}
  \C(f) ( x_1 \otimes\cdots\otimes x_n )
  = \sum_{\pi\in S_n} \sum_{k=0}^\infty \frac{1}{k!} \sum_{n_1+\cdots+n_k=n}
  \frac{(-1)^\epsilon}{n_1!\ldots n_k!}
  f_{(n_1)}(x_{\pi(1)},\ldots) \otimes \cdots \otimes f_{(n_k)}(\ldots,x_{\pi(n_k)}) .
\end{equation*}
The functor $\C(L)$ embeds the category $\LL$ of \Linf-algebras as a
full subcategory of the category of filtered dg cocommutative
coalgebras.

Berglund applies homological perturbation theory to the dg coalgebra
$\C(L)$ to obtain a homotopical perturbation theory for \Linf-algebras
\cite{Berglund}. In this paper, we apply a curved extension of
Berglund's theorem to prove the following.
\begin{theorem}
  \label{main}
  The natural transformation $\gamma_\bullet(L)\to\MC_\bullet(L)$ has a natural
  retraction
  \begin{equation*}
    \rho:\MC_\bullet(L)\to\gamma_\bullet(L) .
  \end{equation*}
\end{theorem}

The morphism $\rho:\MC_\bullet(L)\to\gamma_\bullet(L)$ is an analogue of holonomy for curved
\Linf-algebras. Its construction is explicit, and formulas for $\rho$
could in principle be extracted from the proof. Due to the complexity
of the Dupont homotopy, these formulas are very difficult to work
with: this is the reason that we introduce cubical analogues of the
functors $\MC_\bullet(L)$ and $\gamma_\bullet(L)$ in a sequel to this paper
\cite{cubical}. The analogue of the Dupont homotopy on the $n$-cube
has $n$ terms, while the Dupont homotopy on the $n$-simplex has
$2^{n+1}-2$ terms. The following is the main result of \cite{cubical}.
\begin{theorem}
  There is a natural equivalence of functors $\gamma^\square_\bullet(L)\cong\gamma_\bullet(L)$.
\end{theorem}

Kapranov \cite{Kapranov} considered the following class of curved
\Linf-algebras in the setting of dg Lie algebras.
\begin{definition}
  A curved \Linf-algebra $L$ is \textbf{semiabelian} if $L^{\ge-n}$ is a
  curved \Linf-subalgebra of $L$ for $n>0$.
\end{definition}

Every dg Lie algebra concentrated in degrees $[-1,\infty)$ is semiabelian.
In \cite{cubical}, we identify $\rho$ for $L$ semiabelian with the higher
holonomy of Kapranov \cite{Kapranov} and Bressler et al.\ \cite{BGNT}.

If $L$ is a nilpotent Lie algebra $\Gg$, the $n$-simplices of
$\MC_\bullet(\Gg)$ are the flat $\Gg$-connections over the
$n$-simplex. The simplicial set $\gamma_\bullet(\Gg)$ is naturally equivalent to
the nerve of the pro-nilpotent Lie group $\G(\Gg)$ associated to
$\Gg$, and the function $\rho:\MC_1(\Gg)\to\gamma_1(\Gg)$ is the path-ordered
exponential.

Let $C\Gg$ be the cone of a nilpotent Lie algebra $\Gg$; this is the
dg Lie algebra $0\to\Gg\to\Gg\to0$, equaling $\Gg$ in degrees $0$ and
$-1$, with differential the identity map. An element of $\MC_n(C\Gg)$
is a $\Gg$-connection on the $n$-simplex, without any condition on the
curvature. Since $C\Gg$ is semiabelian, we obtain an identification of
the holonomy $\rho:\MC_2(C\Gg)\to\gamma_2(C\Gg)$ on 2-simplices with the higher
holonomy of a $\Gg$-connection over the 2-simplex.

\section*{Categories of fibrant objects}

A \textbf{category with weak equivalences} $(\CV,\CW)$ is a category
$\CV$, together with a subcategory $\CW\subset\CV$ satisfying the following
axioms.
\begin{enumerate}[label=(W\arabic*),labelwidth=\widthof{(W2)},leftmargin=!]
\item Every isomorphism is a weak equivalence.
\item If $f$ and $g$ are composable morphisms such that $gf$ is a weak
  equivalence, then if either $f$ or $g$ is a weak equivalence, then
  both $f$ and $g$ are weak equivalences.
\end{enumerate}
If $\CV$ is small, the pair $(\CV,\CW)$ has a simplicial localization
$L(\CV,\CW)$. This is a simplicial category with the same objects as
$\CV$ that refines the usual localization $\Ho(\CV)=\CW^{-1}\CV$, in
the sense that the morphisms of the localization are the components of
the simplicial sets of morphisms of $L(\CV,\CW)$.

Categories of fibrant objects, introduced by Brown \cite{Brown} in his
work on simplicial spectra, are a convenient setting in which to study
the simplicial localization; in a category of fibrant objects, the
simplicial set of morphisms in the simplicial localization between two
objects is the nerve of a category of spans.

\begin{definition}
  A \textbf{category of fibrant objects} $(\CV,\CW,\CF)$ is a category
  with weak equivalences $(\CV,\CW)$, together with a subcategory
  $\CF\subset\CV$ of fibrations, satisfying the following axioms. We refer
  to morphisms $f\in\CF\cap\CW$ which are both fibrations and weak
  equivalences as \textbf{trivial fibrations}.
  \begin{enumerate}[label=(F\arabic*),labelwidth=\widthof{(F5)},
    leftmargin=!]
  \item Every isomorphism is a fibration.
  \item Pullbacks of fibrations exist, and are fibrations.
  \item There exists a terminal object $e$ in $\CV$, and any morphism
    with target $e$ is a fibration.
  \item Pullbacks of trivial fibrations are trivial fibrations.
  \item Every morphism $f:X\to Y$ has a factorization
    \begin{equation*}
      \hspace{-\leftmargin}
      \begin{tikzcd}[row sep=0.5em]
        & P \ar[rd,"q"] \\
        X \ar[ru,"s"] \ar[rr,"f"'] & & Y
      \end{tikzcd}
    \end{equation*}
    where $s$ is a weak equivalence and $q$ is a fibration.
  \end{enumerate}
\end{definition}

It follows from the axioms that $\CV$ has finite products. Let $Y$ be
an object of $\CV$. The diagonal $Y\to Y\times Y$ has a factorization into a
weak equivalence followed by a fibration:
\begin{equation*}
  \begin{tikzcd}[row sep=0.5em]
    & PY \ar[dr,"\partial_0\times\partial_1"] & \\
    Y \ar[ur,"s"] \ar[rr] & & Y\times Y
  \end{tikzcd}
\end{equation*}
The object $PY$ is called a \textbf{path space} of $Y$. The proof of
the following lemma shows that the existence of a path space for every
object of $\CV$ is equivalent to Axiom~(F5).
\begin{lemma}[Brown's lemma]
  The weak equivalences of a category of fibrant objects are
  determined by the trivial fibrations: a morphism $f$ is a weak
  equivalence if and only if it factorizes as a composition $qs$,
  where $q$ is a trivial fibration and $s$ is a section of a trivial
  fibration.
\end{lemma}

A functor between categories of fibrant objects is \textbf{exact} if
it preserves fibrations, trivial fibrations, the terminal object, and
pullbacks along fibrations. By Brown's lemma, exact functors also
preserve weak equivalences.

The simplicial set $\Lambda^n_i$ is the union of all faces
$\partial_j\triangle^n$ except the $i$th of the $n$-simplex
$\triangle^n$. A simplicial set $X_\bullet$ is fibrant (or a Kan complex) if the map
\begin{equation*}
  X_n \to \Hom(\Lambda^n_i,X)
\end{equation*}
is surjective for all $0<i\le n$. A \textbf{fibration} of fibrant
simplicial sets is a simplicial morphism $f:X\to Y$ such that the map
\begin{equation*}
  X_n \to \Hom(\Lambda^n_i,X) \times_{\Hom(\Lambda^n_i,Y)} Y_n
\end{equation*}
is surjective for all $i>0$. (The omission of $i=0$ when $n>0$ is
sanctioned by a theorem of Joyal \cite{Joyal}*{Corollary 4.16}.) The
trivial fibrations are the simplicial morphisms $f:X\to Y$ such that the
map
\begin{equation*}
  X_n \to \Hom(\partial\triangle^n,X) \times_{\Hom(\partial\triangle^n,Y)} Y_n
\end{equation*}
is surjective for all $n\ge0$. The full subcategory of fibrant
simplicial sets is a category of fibrant objects $\S$, with functorial
path object $PX$:
\begin{equation*}
  PX_n = \Hom(\triangle^n\times\triangle^1,X) .
\end{equation*}
By the simplicial approximation theorem, a simplicial morphism
$f:X \to Y$ is a weak equivalence if and only if the geometric
realization $|f|:|X|\to|Y|$ is a homotopy equivalence of topological
spaces (Curtis \cite{Curtis}). A skeleton of the subcategory of
fibrant simplicial sets of cardinality less than a fixed infinite
cardinal $\aleph$ is a small category of fibrant objects.

We now show that the category $\LL$ of curved \Linf-algebras is a
category of fibrant objects. If $L$ is a curved \Linf-algebra, $\gr L$
is naturally a filtered complex, with differential
\begin{equation*}
  \delta x=\{x\} \pmod{F^{p+1}L}
\end{equation*}
for $x\in F^pL$.

Denote by $L_\sharp$ the underlying filtered graded vector space of a
curved \Linf-algebra. Denote the linear component $f_{(1)}$ of a
morphism $f:L\to M$ of curved \Linf-algebras by
$\d f:L_\sharp\to M_\sharp$, and by $\gr\d f:\gr L\to\gr M$ the induced morphism of
complexes.
\begin{definition}
  A morphism $f:L\to M$ of curved \Linf-algebras is a \textbf{weak
    equivalence} if
  \begin{equation*}
    \gr \d f:\gr L \to \gr M
  \end{equation*}
  is a quasi-isomorphism.
\end{definition}

The weak equivalences form a subcategory $\CW$ of $\LL$, making it
into a category with weak equivalences; likewise, the category
$\CW\cap\L$ of strict weak equivalences makes $\L$ into a category with
weak equivalences. Note that retracts of weak equivalences are weak
equivalences in both of these categories.

\begin{definition}
  A morphism $f:L \to M$ of curved \Linf-algebras is a
  \textbf{fibration} if $\d f$ is surjective.
\end{definition}

A fibration $f$ is a trivial fibration if and only if the complex
$(\gr K,\delta)$ is contractible, where $K$ is the kernel of $\d f$. Every
isomorphism is a trivial fibration. Note that retracts of fibrations
are fibrations, and that retracts of trivial fibrations are trivial
fibrations.

\begin{lemma}
  \label{fibration}
  A morphism $f:L \to M$ of curved \Linf-algebras is a fibration if and
  only if $\d f$ has a section, that is, a morphism
  $s:M_\sharp \to L_\sharp$ of filtered graded vector spaces such that
  $\d f\circ s$ is the identity of $M_\sharp$.
\end{lemma}
\begin{proof}
  It is clear that the first condition implies that $f$ is a
  fibration. To see the reverse implication, first choose a section
  $\gr s:\gr M\to\gr L$ of the morphism $\gr\d f:\gr L\to\gr M$. Next,
  choose isomorphisms
  \begin{equation*}
    L/F^pL \cong \bigoplus_{q<p} \gr^qL \quad\text{and}\quad
    M/F^pM \cong \bigoplus_{q<p} \gr^qM
  \end{equation*}
  that are compatible with the morphisms
  \begin{equation*}
    \alpha_{p,q} : L/F^qL \to L/F^pL \quad\text{and}\quad \beta_{p,q} : M/F^qM \to M/F^pM
  \end{equation*}
  when $p\le q$. In this way, we obtain sections
  $s_p:M/F^pM \to L/F^pL$ such that
  \begin{equation*}
    \alpha_{p,q}s_q = s_p\beta_{p,q} .
  \end{equation*}
  Take the limit of $s_p$ over $p$ to obtain a section $s:M\to L$.
\end{proof}

The following result was proved by Rogers \cite{Rogers} when the
curvatures of $L$, $M$ and $N$ vanish.
\begin{lemma}
  \label{pullback}
  If $f$ is a fibration, the fibered product $L\times_MN$ of \Linf-algebras
  \begin{equation*}
    \begin{tikzcd}
      L \times_M N \arrow[r,dashrightarrow,"G"] \pullback
      \arrow[d,dashrightarrow,"F"'] & L \dar{f} \\
      N \arrow[r,"g"'] & M
    \end{tikzcd}
  \end{equation*}
  exists. The pullback $F$ of the fibration $f$ may be taken to be a
  strict fibration.
\end{lemma}
\begin{proof}
  Choose a section $s:M_\sharp\to L_\sharp$ of $\d f$. This section induces a
  projection $p=1-s\circ \d f:L_\sharp \to L_\sharp$, with image the kernel of
  $s$. The fibered product is realized on the filtered graded vector
  space $pL\times N$. The morphism $F:pL\times N\to N$ is the strict fibration
  given by the projection to the second factor. The morphism
  $G:pL\times M\to L$ satisfies the equations
  $f\bigl( G_{(0)} \bigr) = g_{(0)}$ and $f\bullet G=g\bullet F$, which may be
  written
  \begin{multline*}
    f_{(1)}\bigl( G_{(n)}(\zeta_1,\ldots,\zeta_n) \bigr)
    + \sum_{k=1}^\infty \frac{1}{k!} \Biggl( f_{(k+1)}\bigl(
    G_{(n)}(\zeta_1,\ldots,\zeta_n) , G_{(0)},\ldots,G_{(0)} \bigr)
    + \sum_{\pi\in S_n} \\
    \sum_{\substack{n_1+\cdots+n_k=n\\0\le n_i<n}}
    \frac{(-1)^\epsilon}{n_1!\ldots n_k!} f_{(k)}\bigl( G_{(n_1)}(\zeta_{\pi(1)},\ldots),\ldots,
    G_{(n_k)}(\ldots,\zeta_{\pi(n)}) \bigr) \Biggr) = g_{(k)}(y_1,\ldots,y_k) ,
  \end{multline*}
  where $\zeta_i\in pL \times N$. These equations have a unique solution
  satisfying the gauge conditions
  \begin{equation*}
    pG_{(n)}(\zeta_1,\ldots,\zeta_n) =
    \begin{cases}
      p\zeta_1 , & n=1 , \\
      0 , & \text{otherwise.}
    \end{cases}
  \end{equation*}
  The element $G_{(0)}=z\in F^1L$ is determined by the equation
  \begin{equation*}
    z + \sum_{k=2}^\infty \frac{1}{k!} \, sf_{(k)}( z,\ldots,z ) = sg_{(0)} .
  \end{equation*}
  The element $G_{(1)}(\zeta)=z\in L$ is determined by the equation
  \begin{equation*}
    z + \sum_{k=1}^\infty \frac{1}{k!} \, sf_{(k+1)}\bigl( z,
    G_{(0)},\ldots,G_{(0)} \bigr) = x + sg_{(1)}(F\zeta) .
  \end{equation*}
  The element $G_{(n)}(\zeta_1,\ldots,\zeta_n)=z\in L$ is determined by the
  equation
  \begin{multline*}
    z + \sum_{k=1}^\infty \frac{1}{k!} \, sf_{(k+1)}\bigl( z ,
    G_{(0)},\ldots,G_{(0)} \bigr) = sg_{(n)}(F\zeta_1,\ldots,F\zeta_n) \\
    - \sum_{\pi\in S_n} \sum_{k=2}^\infty \frac{(-1)^\epsilon}{k!}
    \sum_{\substack{n_1+\cdots+n_k=n\\0\le n_i<n}} \frac{1}{n_1!\ldots n_k!} \,
    sf_{(k)}\bigl( G_{(n_1)}(\zeta_{\pi(1)},\ldots),\ldots,G_{(n_k)}(\ldots,\zeta_{\pi(n)}) \bigr) .
  \end{multline*}

  The bracket $\{\!\{ \zeta_1 , \ldots , \zeta_n \}\!\}$ on
  $L\times_MN$ is characterized by its compatibility with $F$ and $G$:
  compatibility with $F$ implies that
  $F\{\!\{ \zeta_1 , \ldots , \zeta_n \}\!\} = \{ F\zeta_1 , \ldots , F\zeta_n \}$, while
  compatibility with $G$, namely the equation,
  \begin{multline}
    \label{G}
    \tag{$\ast$}
    \sum_{\pi\in S_n} \sum_{k=0}^n \frac{(-1)^\epsilon}{k!(n-k)!} \,
    G_{(n-k+1)}( \{\!\{ \zeta_{\pi(1)},\ldots,\zeta_{\pi(k)} \}\!\},\zeta_{\pi(k+1)},\ldots,\zeta_{\pi(n)}) \\
    = \sum_{\pi\in S_n} \sum_{k=0}^\infty \frac{(-1)^\epsilon}{k!}
    \sum_{n_1+\cdots+n_k=n} \frac{1}{n_1!\ldots n_k!} \,
    \{ G_{(n_1)}(\zeta_{\pi(1)},\ldots),\ldots,G_{(n_k)}(\ldots,\zeta_{\pi(n)}) \} ,
  \end{multline}
  identifies the result of applying $p$ to the right-hand side of this
  equation with $p\{\!\{ \zeta_1,\ldots,\zeta_n \}\!\}$.

  To show that $G$ is a morphism of curved \Linf-algebras, we must
  prove \eqref{G}; in light of the definition of $G$, this amounts to
  the equation
  \begin{align}
    \tag{$\ast\ast$}
    & \label{pG}
    \sum_{\pi\in S_n} \sum_{j=0}^n \frac{(-1)^\epsilon}{j!(n-j)!} \,
    (1-p) G_{(n-j+1)}( \{\!\{ \zeta_{\pi(1)},\ldots,\zeta_{\pi(j)} \}\!\},
    \zeta_{\pi(j+1)},\ldots,\zeta_{\pi(n)}) \\
    \notag
    & \qquad = \sum_{\pi\in S_n} \sum_{k=0}^\infty \frac{(-1)^\epsilon}{k!}
    \sum_{n_1+\cdots+n_k=n} \frac{1}{n_1!\ldots n_k!} \, (1-p)
    \{ G_{(n_1)}(\zeta_{\pi(1)},\ldots),\ldots,G_{(n_k)}(\ldots,\zeta_{\pi(n)}) \} .
  \end{align}
  The equation $f\bullet G=g\bullet F$ along with $g$ and $F$ being morphisms of
  curved \Linf-algebras shows that
  \begin{multline*}
    \sum_{\pi\in S_n} \sum_{k=0}^\infty \frac{(-1)^\epsilon}{k!} \sum_{n_0+\cdots+n_k=n}
    \frac{1}{n_0! \ldots n_k!} \\
    \sum_{j=0}^{n_0} \binom{n_0}{j} \,
    f_{(k+1)}\bigl( G_{(n_0)}(\{\zeta_{\pi(1)},\ldots,\zeta_{\pi(j)}\},\zeta_{\pi(j+1)},\ldots),
    \ldots , G_{(n_k)}(\ldots,\zeta_{\pi(n)}) \bigr) \\
    = \sum_{\pi\in S_n} \sum_{j=0}^n \frac{(-1)^\epsilon}{j!(n-j)!} \,
    g_{(n-j+1)}\bigl(\{F\zeta_{\pi(1)},\ldots,F\zeta_{\pi(j)}\}, F\zeta_{\pi(j+1)} , \ldots , F\zeta_{\pi(m)}
    \bigr) .
  \end{multline*}
  Applying $s$ to both sides of this equation gives \eqref{pG}.

  To show that $L\times_MN$ is a pullback, we must show the existence of
  the morphism $\epsilon$ for any commutative diagram of curved
  \Linf-algebras of the form
  \begin{equation*}
    \begin{tikzcd}
      A \arrow[rrd,bend left,"\lambda"] \arrow[ddr,bend right,"\nu"']
      \drar[dotted,"\epsilon"] & & \\
      & L\times_MN \rar{G} \dar[']{F} \pullback & L \dar{f} \\
      & N \rar[']{g} & M
    \end{tikzcd}
  \end{equation*}
  The morphism $\epsilon$ has components
  \begin{equation*}
    \epsilon_{(n)}(z_1,\ldots,z_n) = p\lambda_{(n)}(z_1,\ldots,z_n) \times \nu_{(n)}(z_1,\ldots,z_n)
    . \qedhere
  \end{equation*}
\end{proof}

\begin{corollary}
  Every fibration $f:L\to M$ is isomorphic to a strict fibration $F$.
\end{corollary}
\begin{proof}
  Apply the theorem with $g$ equal to the identity of $M$.
\end{proof}


\begin{proposition}
  The categories $\L$ and $\LL$ of \Linf-algebras are categories of
  fibrant objects, and the inclusion $\L\hookrightarrow\LL$ is an exact functor.
\end{proposition}
\begin{proof}
  The proofs that $\L$ and $\LL$ are categories of fibrant objects are
  identical, so we focus on $\LL$.

  We have already seen that the object $0\in\LL$ is terminal. It is
  clear that every morphism of $\LL$ with target $0$ is a
  fibration. By Lemma~\ref{pullback}, fibrations have pullbacks, and
  the pullback of a fibration is a fibration. Let $f:L\to M$ be a
  fibration, and let $K\subset L$ be the kernel of $\d f$. Then $f$ is a
  trivial fibration if and only if $(\gr K,\delta)$ is contractible; we
  conclude that the pullback of a trivial fibration is a trivial
  fibration.

  The diagonal morphism $L\to L\times L$ factors through
  $\DR_1\ohat L\to L\times L$; this is the fibration taking
  $a(t)+b(t)dt\in L[t,dt]$ to $f(0)\times f(1)$. The inclusion of $L$ in
  $\DR_1\ohat L$ is a strict morphism and a weak equivalence: it is a
  section of the weak equivalences
  $\partial_0,\partial_1:\DR_1\ohat L\to L$ given by projecting
  $L\times L$ to the first and second factors.
\end{proof}

The same proof shows that a skeleton of the subcategory of curved
\Linf-algebras of dimension less than a fixed infinite cardinal
$\aleph>\aleph_0$ is a small category of fibrant objects. (The case
$\aleph=\aleph_0$, follows as in Rogers \cite{Rogers} using an \Linf-structure
$W_1\otimes L$ constructed using Theorem~\ref{Berglund}.)  In the remainder
of this paper, $\LL$ will denote this small category, and $\L$ its
small subcategory of strict morphisms.

Using the description of morphisms in the category $\LL_\pi$ in terms of
the twisting cochain associated to a cofibrant resolution of the Lie
operad, it seems likely that $\LL_\pi$ is also a category of fibrant
objects.

\begin{definition}
  An exact functor $F:C\to D$ between categories of fibrant objects
  satisfies the \textbf{Waldhausen Approximation Property} if
  \begin{enumerate}[1)]
  \item $F$ reflects weak equivalences ($f:x\to y$ is a weak equivalence
    if $F(f):F(x)\to F(y)$ is a weak equivalence);
  \item every morphism $f:z\to F(y)$ in $D$, there is a morphism
    $h:x\to y$ in $C$ and a weak equivalence $g:z\to F(x)$ in $D$ such
    that $f=F(h)g$.
  \end{enumerate}
\end{definition}

Cisinski \cite{Cisinski} proves that an exact functor induces a weak
equivalence of simplicial localizations if it satisfies the Waldhausen
Approximation property.
\begin{proposition}
  The inclusion $\L\hookrightarrow\LL$ satisfies the Waldhausen Approximation
  Property.
\end{proposition}
\begin{proof}
  The first condition is obvious. The second is proved for a morphism
  $f:L \to M$ as follows. The curved \Linf-algebra has a dg Lie
  resolution $p:\tilde{L}\to L$ such that $p$ is a trivial fibration and
  $fp=\tilde{f}$ is a strict morphism. (We may take $\tilde{L}$ to be
  the space of primitive elements in the cobar construction of
  $\C(L)$; this is a dg Hopf algebra because $\C(L)$ is
  cocommutative.) We take $g$ to be the inclusion of $L$ into
  $L\times_{\tilde{L}}\bigl(\DR_1\ohat\tilde{L}\bigr)$, and $h$ to be the
  projection from $L\times_{\tilde{L}}\bigl(\DR_1\ohat\tilde{L}\bigr)$ to
  $\tilde{L}$.
\end{proof}

The following result justifies our definition of trivial fibrations.
\begin{theorem}
  \label{surjective}
  If $f:L\to M$ is a trivial fibration, the map $f:\MC(L)\to\MC(M)$ is
  surjective.
\end{theorem}
\begin{proof}
  By universality, it suffices to construct a Maurer--Cartan element of
  the curved \Linf-algebra $L\times_M0$ of Lemma~\ref{fibration} associated
  to the diagram
  \begin{equation*}
    \begin{tikzcd}
      L \times_M 0 \rar \pullback \dar & L \dar{f} \\
      0 \arrow[r,"y"'] & M
    \end{tikzcd}
  \end{equation*}
  In other words, we may assume in the proof of the theorem that
  $M=0$, and we are reduced to showing that a contractible
  \Linf-algebra $L$ has a Maurer--Cartan element.

  Since $L$ is contractible, the differential $\delta_i$ on $\gr^iL$
  induced by $x\to\{x\}$ has a contracting homotopy
  $h_i:\gr^i L\to\gr^i L$, satisfying $\delta_ih_i+h_i\delta_i=1$. Replacing
  $h_i$ by $h_i\delta_ih_i$, we may assume that $h_i^2=0$. Lift $h$ to
  $L$, by choosing a splitting of the filtration on $L$, that is,
  isomorphisms
  \begin{equation*}
    L/F^pL \cong \bigoplus_{i<p} \gr^iL
  \end{equation*}
  as in the proof of Lemma~\ref{fibration}, and defining $h$ to be the
  map on $L$ induced by the maps
  \begin{equation*}
    h_p = \bigoplus_{i<p} h^i
  \end{equation*}
  on $L/F^pL$. If $x\in F^pL$, we have
  \begin{equation*}
    x - h\{x\} - \{hx\} \in F^{p+1}L .
  \end{equation*}
  We show that there is a (unique) Maurer--Cartan element $x\in\MC(L)$
  such that $hx=0$. Applying $h$ to the Maurer--Cartan equation, we
  obtain the (curved) Kuranishi equation
  \begin{align*}
    x &= x - \sum_{n=0}^\infty \frac{1}{n!} \, h\,\{x^{\otimes n}\} \\
      &= - h\{\} + \bigl( x - h\{x\} \bigr) - \sum_{n=2}^\infty \frac{1}{n!}
        \, h\,\{x^{\otimes n}\} = \Phi(x) .
  \end{align*}

  If $x$ and $y$ are two solutions of this equation and $x-y\in F^pL$,
  then
  \begin{equation*}
    x - y = \bigl( x - h\{x\} \bigr) - \bigl( y - h\{y\} \bigr)
    - \sum_{m+n>0} \frac{1}{(m+n+1)!} \, h\,\{x-y,x^{\otimes m},y^{\otimes n}\}
    \in F^{p+1}L ,
  \end{equation*}
  and hence $x=y$. Thus, solutions to this equation are unique.

  A similar argument shows that a solution exists: set $x_0=0$ and
  $x_{k+1}=\Phi(x_k)$. We have
  \begin{equation*}
    x_{k+1} - x_k = \bigl( x_k - h\{x_k\} \bigr)
    - \bigl( x_{k-1} - h\{x_{k-1}\} \bigr)
    - \sum_{m+n>0}
    \frac{1}{(m+n+1)!} \, h\,\{x_k-x_{k-1},x_k^{\otimes m},x_{k-1}^{\otimes n}\} .
  \end{equation*}
  We see by induction that $x_k-x_{k-1}\in F^kL$, and hence by
  completeness of the filtration on $L$ that the limit
  $x_\infty=\lim_{k\to\infty}x_k$ exists.
  
 Then $x_\infty=\Phi(x_\infty)$, and it remains to show that $x_\infty\in\MC(L)$. Let
  \begin{equation*}
    z = \sum_{n=0}^\infty \frac{1}{n!} \{ x_\infty^{\otimes n} \}
  \end{equation*}
  be the curvature of $x_\infty$. The Kuranishi equation implies that
  \begin{equation*}
    z = \bigl( z - h\{z\} \bigr) - \sum_{n=1}^\infty \frac{1}{n!} \,
    h \{ x_\infty^{\otimes n} , z \} = \Psi(z) .
  \end{equation*}
  The fixed-point equation $z=\Psi(z)$ has a unique solution $z=0$,
  showing that $x_\infty\in\MC(L)$.
\end{proof}

The proof that $\MC_\bullet(L)$ is fibrant relies on the following extension
lemma of Bousfield and Gugenheim \cite{BG}*{Corollary 1.2}. If $X$ is
a simplicial set, the dg commutative algebra of differential forms on
$X$ is the limit
\begin{equation*}
  \DR(X) = \int_{[n]\in\triangle} \Hom(X_n,\DR_n) .
\end{equation*}
The set $\Hom(X,\MC_\bullet(L))$ of simplicial maps from $X_\bullet$ to the nerve
is naturally equivalent to $\MC(\DR(X)\ohat L)$.
\begin{lemma}
  \label{extension}
  If $i:X\to Y$ is a cofibration of simplicial sets (that is, $i_n$ is a
  monomorphism for all $n$), the morphism
  $(i^*)_\sharp:\DR(Y)_\sharp\to\DR(X)_\sharp$ has a section $\sigma:\DR(X)_\sharp\to\DR(Y)_\sharp$.
\end{lemma}
\begin{proof}
  By induction, it suffices to prove the result for the generating
  cofibrations $\partial\triangle^n\to\triangle^n$, $n\ge0$. We give a formula for a section
  $\sigma_n:\DR(\partial\triangle^n)_\sharp\to\DR(\triangle^n)_\sharp$:
  \begin{equation*}
    \sigma_n\omega = \sum_{i=0}^n t_i \sum_{\emptyset\ne J\subset\{0,\ldots,\hat\imath,\ldots,n\}} (-1)^{|J|-1}
    \sigma_{i,J}^*\omega ,
  \end{equation*}
  where $\sigma_{i,J}:\triangle^n\to\triangle^n$ is the affine morphism that takes the
  vertices $e_j\in\triangle^n$, $j\in J$, to $e_i$, leaving the remaining vertices
  fixed. (This formula comes from the proof of
  \cite{Linf}*{Lemma~3.2}, which was suggested to the author by a
  referee of that article.) Consider the restriction of
  $\sigma_n\omega$ to $\partial_j\Delta^n=\{t_j=0\}$. For $i\ne j$, the sum
  \begin{equation*}
    \sum_{\emptyset\ne J\subset\{0,\ldots,\hat\imath,\ldots,n\}} (-1)^{|J|-1} \sigma_{i,J}^*\omega|_{t_j=0}
  \end{equation*}
  equals $\omega|_{t_j=0}$, since for $J\subset\{0,\ldots,n\}\setminus\{i,j\}$, we have
  \begin{equation*}
    \sigma_{i,J}^*\omega|_{t_j=0} = \sigma_{i,J\cup\{i\}}^*\omega|_{t_j=0} \bigr) ,
  \end{equation*}
  and thus all of the terms cancel except
  $\sigma_{j,\{j\}}^*\omega|_{t_j=0} = \omega|_{t_j=0}$. Taking the sum
  over $i$ in $\sigma_n^*\omega|_{t_j=0}$, we obtain
  $\omega|_{t_j=0}$. That is, the restriction of
  $\sigma_n^*\omega$ to $\partial_j\Delta^n$ equals $\omega$.
\end{proof}

\begin{corollary}
  \label{Fibration}
  If $f:L\to M$ is a fibration of curved \Linf-algebras and $i:X\to Y$ is
  a cofibration of simplicial sets, the strict morphism
  \begin{equation*}
    \epsilon : \DR(Y) \ohat L \to \bigl( \DR(X)\ohat L \bigr)
    \times_{\DR(X)\ohat M} \bigl( \DR(Y) \ohat M \bigr)
  \end{equation*}
  is a fibration.
\end{corollary}
\begin{proof}
  Let $K\subset L$ be the kernel of $\d f:L\to M$. We have an identification
  of filtered graded vector spaces
  \begin{equation*}
    \bigl( \bigl( \DR(X)\ohat L \bigr) \times_{\DR(X)\ohat M} \bigl(
    \DR(Y) \ohat M \bigr) \bigr)_\sharp \cong
    \bigl( \DR(X)\ohat K \bigr)_\sharp \oplus \bigl( \DR(Y) \ohat M \bigr)_\sharp .
  \end{equation*}
  By Lemma~\ref{extension}, this morphism has a section $(\sigma\otimes1) \oplus 1$.
\end{proof}

\begin{proposition}
  \label{exact}
  The functor $\MC_\bullet(L)$ is an exact functor from the category
  $\LL$ of curved \Linf-algebras to the category $\S$ of fibrant
  simplicial sets.
\end{proposition}
\begin{proof}
  It is clear that $\MC_\bullet(L)$ takes the terminal curved \Linf-algebra
  $0$ to the terminal simplicial set $\ast$, and fibered products with
  fibrations to fibered products. It remains to show that if
  $f:L\to M$ is a (trivial) fibration, the morphism
  $\MC_\bullet(f):\MC_\bullet(L)\to\MC_\bullet(M)$ of simplicial sets is a (trivial)
  fibration of simplicial sets.

  We first show that $\MC_\bullet(f):\MC_\bullet(L)\to\MC_\bullet(M)$ is a fibration if
  $f:L\to M$ is. By Theorem~\ref{surjective}, this follows once we
  show that for each $0<i\le n$, the strict morphism of curved
  \Linf-algebras
  \begin{equation}
    \label{trivial:horn}
    \epsilon : \DR_n \ohat L \to \bigl( \DR(\Lambda^n_i)\ohat L \bigr)
    \times_{\DR(\Lambda^n_i)\ohat M} \bigl( \DR_n \ohat M \bigr)
  \end{equation}
  is a trivial fibration. It is a fibration by
  Corollary~\ref{Fibration}. It remains to show that it is a weak
  equivalence.

  Consider the commutative diagram
  \begin{equation*}
    \begin{tikzcd}[column sep=huge]
      \DR_n \ohat L \arrow[rrd,bend left=5,"\alpha"] \arrow[ddr,bend right=15]
      \drar["\epsilon"] & & \\
      & \bigl( \DR(\Lambda^n_i) \ohat L \bigr) \times_{\DR(\Lambda^n_i)\ohat M} \bigl( \DR_n
      \ohat M \bigr) \rar["\beta"'] \dar \pullback & \DR(\Lambda^n_i) \ohat L \dar \\
      & \DR_n \ohat M \rar["\gamma"'] & \DR(\Lambda^n_i) \ohat M
    \end{tikzcd}
  \end{equation*}
  The contracting homotopy $h^n_i\otimes 1$ on $\DR_n\ohat L$ satisfies
  \begin{equation*}
    (d\otimes1+1\otimes\delta)h^n_i + h^n_i(d\otimes1+1\otimes\delta) = 1 - \epsilon^n_i\otimes 1 ,
  \end{equation*}
  and its restriction to $\DR(\Lambda^n_i)\ohat L$ satisfies the same
  equation. This shows that the downward arrows in the commutative
  diagram
  \begin{equation*}
    \begin{tikzcd}
      \DR_n \ohat L \ar[rr,"\alpha"] \ar[rd,"\epsilon^n_i\otimes1"'] & &
      \DR(\Lambda^n_i) \ohat L \arrow[ld,"\epsilon^n_i\otimes1"] \\
      & L
    \end{tikzcd}
  \end{equation*}
  are quasi-isomorphisms. It follows that $\alpha$ is a weak equivalence,
  and hence a trivial fibration. The same argument with $L$ replaced
  by $M$ shows that $\gamma$ is a trivial fibration, and hence that its
  pullback $\beta$ is a trivial fibration. Finally, we see that
  $\epsilon$ is a weak equivalence.

  It remains to show that
  $\MC_\bullet(f):\MC_\bullet(L)\to\MC_\bullet(M)$ is a trivial fibration if
  $f:L\to M$ is. By Theorem~\ref{surjective}, this follows once we show
  that for each $n\ge0$, the strict morphism of curved \Linf-algebras
  \begin{equation}
    \label{trivial:boundary}
    \epsilon : \DR_n\ohat L \to \bigl( \DR(\partial\triangle^n)\ohat L \bigr)
    \times_{\DR(\partial\triangle^n)\ohat M} \bigl( \DR_n\ohat M \bigr)
  \end{equation}
  is a trivial fibration. It is a fibration by
  Lemma~\ref{fibration}. It remains to show that it is a weak
  equivalence.

  Consider the commutative diagram
  \begin{equation*}
    \begin{tikzcd}[column sep=huge]
      \DR_n \ohat L \arrow[rrd,bend left=5] \arrow[ddr,bend
      right=15,"\alpha"'] \drar["\epsilon"] & & \\
      & \bigl( \DR(\partial\triangle^n) \ohat L \bigr) \times_{\DR(\partial\triangle^n)\ohat M} \bigl(
      \DR_n \ohat M \bigr) \rar \dar["\beta"] \pullback & \DR(\partial\triangle^n) \ohat L
      \dar["\gamma"] \\
      & \DR_n \ohat M \rar & \DR(\partial\triangle^n) \ohat M
    \end{tikzcd}
  \end{equation*}
  Since $f$ is a trivial fibration, we see that $\alpha$, $\beta$ and
  $\gamma$ are as well. We conclude that $\epsilon$ is a weak equivalence, and
  hence a trivial fibration.
\end{proof}

The functor $\MC_\bullet(L)$ restricts to an exact functor from the category
of curved \Linf-algebras of dimension less than $\aleph$ to the category of
Kan complexes of cardinality less than $|\F|^\aleph$.


\section*{Homotopical perturbation theory for curved \Linf-algebras} 

\begin{definition}
  A \textbf{contraction} of filtered complexes from $(V,\D)$ to
  $(W,d)$ consists of filtered morphisms of complexes $p:V\to W$ and
  $i:W\to V$ and a map $h:V\to V[-1]$, compatible with the filtration,
  such that
  \begin{align*}
    ip + \D h + h\D & = 1_W , &
    pi & = 1_V , &
    h^2 & = ph = hi = 0
  \end{align*}
\end{definition}

Up to isomorphism, the contraction is determined by the graded vector
space $V$, the differential $\D$ and the map $h$: the complex $(W,d)$
may be identified with the kernel of the morphism $\D h+h\D:V\to V$, the
map $i$ is the inclusion of this kernel in $V$, and $p$ is the
projection from $V$ to the kernel. Contractions were called gauges in
\cite{Linf}.

Let $h$ be a contraction of filtered complexes from $(V,\D)$ to
$(W,d)$. A Maurer-Cartan element $\mu\in \End(V)$ such that $1+\mu h$ (and
hence $1+h\mu$) is invertible gives rise to a new contraction, by the
standard formulas
\begin{align*}
  \D_\mu &= \D+\mu , & h_\mu &= (1+h\mu)^{-1}h , &
   d_\mu &= d + p(1+\mu h)^{-1}\mu i \\
  & & i_\mu &= (1+h\mu)^{-1} i , & p_\mu &= p(1+\mu h)^{-1} .
\end{align*}

Unless its curvature vanishes, a curved \Linf-algebra does not have an
underlying filtered complex. For this reason, the differential $\D$
must be additional data in the definition of a contraction for curved
\Linf-algebras.
\begin{definition}
  A \textbf{contraction} of a curved \Linf-algebra $L$ is a
  contraction between filtered complexes $(V,\D)$ and $(W,d)$ and an
  isomorphism of filtered graded vector spaces $L_\sharp \cong V_\sharp$ such that
  the induced differential on $L$, which we denote by $\D$, satisfies
  $\{x\}-\D x\in F^{p+1}L$ for $x\in F^pL$.
\end{definition}


In \cite{Berglund}, Berglund develops homological perturbation theory
for $\infty$-algebras over general Koszul operads. His approach extends to
curved \Linf-algebras, as we now explain. See Dotsenko et al.\
\cite{DSV} for an alternative approach.

By analogy with the \emph{tensor trick} of Gugenheim et al.\
\cite{GLS}, associate to a contraction of filtered complexes a
contraction $(\C(V),\C(W),\p,\i,\h)$, where $\C(V)$ and $\C(W)$ have
the differentials $\D$ and $d$ induced by the differentials $\D$ and
$d$ on $V$ and $W$,
\begin{align*}
  \p & = \bigoplus_{n=0}^\infty p^{\otimes n} , &
  \i & = \bigoplus_{n=0}^\infty i^{\otimes n} ,
\end{align*}
are the morphisms of coalgebras induced by $p$ and $i$, and the $\h$ is
the symmetrization of the homotopy
\begin{equation*}
  \bigoplus_{n=0}^\infty \sum_{k=1}^n (ip)^{k-1} \otimes h \otimes 1^{n-k}
\end{equation*}
on the tensor coalgebra, given by the explicit formula
\begin{equation*}
  \h = \bigoplus_{n=0}^\infty \frac{1}{n} \sum_{k=1}^n
    \sum_{\epsilon_1,\ldots,\epsilon_{n-1}\in\{0,1\}} \binom{n-1}{\sum\epsilon_i}^{-1}
  (ip)^{\epsilon_1} \otimes \ldots \otimes (ip)^{\epsilon_{k-1}} \otimes h \otimes (ip)^{\epsilon_k} \otimes \ldots \otimes (ip)^{\epsilon_{n-1}} .
\end{equation*}
The following lemma is due to Berglund~\cite{Berglund}.
\begin{lemma}
  \label{fh}
  We have $(\p\otimes\p)\nabla\h= 0$, $(\h\otimes\p)\nabla\h=(\p\otimes\h)\nabla\h= 0$, and
  $(\h\otimes\h)\nabla\h=0$.
\end{lemma}
\begin{proof}
  We have $(\p\otimes\p)\nabla\h=\nabla\p\h=0$. The remaining three identities follow
  from the explicit formulas for $\p$ and $\h$.
\end{proof}

The Maurer--Cartan element $\mu=\delta-\D$ on $\C(L)$,
\begin{align*}
  \mu(x_1\otimes\ldots\otimes x_k) &= \{\}\otimes x_1\otimes\ldots\otimes x_k \\
  &+ \sum_{i=1}^k (-1)^{|x_1|+\cdots+|x_{i-1}|} x_1\otimes\ldots\otimes x_{i-1}\otimes \bigl(
  \{x_i\}-\D x_i \bigr) \otimes x_{i+1}\otimes\ldots\otimes x_k \\
  &+ \frac{1}{k!} \sum_{\pi\in S_k} \sum_{\ell=2}^k (-1)^\epsilon
  \binom{k}{\ell} \, \{x_{\pi(1)},\ldots,x_{\pi(\ell)}\} \otimes x_{\pi(\ell+1)}\otimes\ldots\otimes x_{\pi(k)} ,
\end{align*}
satisfies $\D_\mu=\delta$. The formulas of homological perturbation theory
yield a differential $d_\mu$ on $\C(W)$ and morphisms of complexes
$\p_\mu:\C(L)\to\C(W)$ and $\i_\mu:\C(W)\to\C(L)$. The following theorem is
the analogue for \Linf-algebras of results of Gugenheim et al.\
\cite{GLS} for \Ainf-algebras.
\begin{theorem}
  \label{Berglund}
  The linear maps $\p_\mu:\C(L)\to\C(W)$ and $\i_\mu:\C(W)\to\C(L)$ are
  morphisms of filtered graded cocommutative coalgebras
  \begin{align*}
    (\p_\mu \otimes \p_\mu) \nabla &= \nabla \p_\mu , & (\i_\mu \otimes \i_\mu) \nabla &= \nabla \i_\mu .
  \end{align*}
  The differential $d_\mu$ is a coderivation of $\C(W)$.
\end{theorem}
\begin{proof}
  The proof follows Berglund \cite{Berglund}. We have
  $\p_\mu=\p-\p_\mu\mu\h$, hence $\p_\mu\i=\p\i=1$. It follows that
  \begin{equation*}
    (\p_\mu\otimes\p_\mu)\nabla\i = (\p_\mu\i\otimes\p_\mu\i)\nabla = \nabla .
  \end{equation*}
  We also have
  \begin{equation*}
    (\p_\mu \otimes \p_\mu) \nabla \h = \bigl( \p\otimes\p - (\p_\mu\mu\otimes1)(\h\otimes\p) -
    (1\otimes\p_\mu\mu)(\p\otimes\h) - (\p\mu\otimes\p\mu)(\h\otimes\h) \bigr) \nabla \h ,
  \end{equation*}
  which vanishes by Lemma~\ref{fh}, proving that
  $(\p_\mu \otimes \p_\mu) \nabla \h = 0$. It follows from this equation that
  \begin{align*}
    0 &= ( d_\mu\otimes1 + 1\otimes d_\mu ) (\p_\mu\otimes \p_\mu)\nabla\h + (\p_\mu\otimes \p_\mu)\nabla\h\D \\
      &= (\p_\mu\otimes \p_\mu)( \D_\mu\otimes1+ 1\otimes \D_\mu )\nabla\h + (\p_\mu\otimes \p_\mu)\nabla\h\D \\
      &= (\p_\mu\otimes \p_\mu)\nabla(\D_\mu\h+\h\D) \\
      &= (\p_\mu\otimes \p_\mu)\nabla \mu\h + (\p_\mu\otimes \p_\mu)\nabla - (\p_\mu\otimes \p_\mu)\nabla\i\p \\
      &= (\p_\mu\otimes \p_\mu)\nabla \mu\h + (\p_\mu\otimes \p_\mu)\nabla - (\p_\mu\i\otimes
        \p_\mu\i)\nabla\p_\mu(1+\mu\h) \\
      &= (\p_\mu\otimes \p_\mu)\nabla \mu\h + (\p_\mu\otimes \p_\mu)\nabla - \nabla\p_\mu(1+\mu\h) .
  \end{align*}
  This proves the formula
  \begin{equation*}
    (\p_\mu\otimes \p_\mu)\nabla(1+\mu\h) = \nabla \p_\mu(1+\mu\h) .
  \end{equation*}
  Since $1+\mu\h$ is invertible, we conclude that $\p_\mu$ is a morphism.

  We turn to $\i_\mu$. It is seen, by induction on $n$, that the
  restriction of $(\h\mu)^k\i\p$ to $(sL)^{\otimes n}_{S_n}\subset \C(L)$ equals
  \begin{equation*}
    (\h\mu)^k\i\p = \sum_{1\le i_1\le\ldots\le i_k\le n} \i\p^{\otimes(i_1-1)}\otimes h\mu \otimes
    \i\p^{\otimes(i_2-i_1-1)}\otimes h\mu \otimes \ldots \otimes h\mu \otimes \i\p^{\otimes(n-i_k)} ,
  \end{equation*}
  and hence that
  \begin{equation*}
    \nabla(\h\mu)^k\i\p = \sum_{i=0}^k \bigl( (\h\mu)^i\otimes(\h\mu)^{k-i} \bigr) \nabla \i\p .
  \end{equation*}
  It follows that
  \begin{align*}
    \nabla\i_\mu &= \nabla(1+\h\mu)^{-1} \i = \nabla(1+\h\mu)^{-1} \i\p\i \\
    &= \bigl( (1+\h\mu)^{-1} \otimes (1+\h\mu)^{-1} \bigr) \nabla \i \\
    &= \bigl( (1+\h\mu)^{-1} \otimes (1+\h\mu)^{-1} \bigr) ( \i \otimes \i ) \nabla = (
      \i_\mu \otimes \i_\mu ) \nabla .
  \end{align*}

  To show that $d_\mu=d+\p_\mu\mu\i$ is a graded coderivation, it suffices
  to show that $\p_\mu\mu\i$ is. We have
  \begin{align*}
    \nabla\p_\mu\mu\i &= (\p_\mu\otimes\p_\mu)(\mu\otimes1+1\otimes\mu)(\i\otimes\i) \nabla \\
             &= ( \p_\mu\mu\i \otimes \p_\mu\i + \p_\mu\i \otimes \p_\mu\mu\i ) \nabla \\
             &= ( \p_\mu\mu\i \otimes 1 + 1 \otimes \p_\mu\mu\i ) \nabla .
               \qedhere
  \end{align*}
\end{proof}

Denote the curved \Linf-algebra with underlying filtered graded vector
space $W$ associated to the codifferential $d_\mu$ on $\C(W)$ by
$\breve{L}$. Then $\p_\mu$ and $\i_\mu$ induce \Linf-morphisms $p_\mu$ from
$L$ to $\breve{L}$ and $i_\mu$ from $\breve{L}$ to $L$.

\begin{proposition}
  The morphism $\MC(p_\mu):\MC(L)\to\MC(\breve{L})$ restricts to a
  bijection from $\MC(L,h)$ to $\MC(\breve{L})$, with inverse
  $\MC(i_\mu)$.
\end{proposition}
\begin{proof}
  Given $x\in L^0$, denote by $\exp(x)$ the element
  \begin{equation*}
    \exp(x) = \sum_{n=0}^\infty \frac{1}{n!} \, x^{\otimes n} \in \C(L) .
  \end{equation*}
  The Maurer-Cartan equation for $x$ is equivalent to the equation
  \begin{equation*}
    \D_\mu\exp(x) = 0 .
  \end{equation*}
  We also have $\exp(\MC(p_\mu)x)=\p_\mu\exp(x)$, $x\in L^0$, and
  $\exp(\MC(i_\mu)y)=\i_\mu\exp(y)$, $y\in M^0$.

  If $hx=0$, we have $\h\exp(x)=0$, and hence
  \begin{equation*}
    \exp(\MC(p_\mu)x) = \p_\mu\exp(x) = \p(1+\mu\h)^{-1}\exp(x) = \p\exp(x)
    = \exp(p(x)) .
  \end{equation*}
  That is, $\MC(p_\mu)x=p(x)$.

  Conversely, if $y\in\MC(\breve{L})$, then
  \begin{equation*}
    \h\exp(\MC(i_\mu)y) = \h\i_\mu\exp(y) = \h(1+\h\mu)^{-1}\i\exp(y) = 0 ,
  \end{equation*}
  and it follows that $h\MC(i_\mu)y=0$. Thus $\MC(i_\mu)$ maps
  $\MC(\breve{L})$ into $\MC(L,h)$.

  If $x\in\MC(L,h)$, we have
  \begin{align*}
    \exp(\MC(i_\mu)\MC(p_\mu)x))
    &= \i_\mu \p_\mu \exp(x) = \bigl( 1 - \D_\mu\h_\mu - \h_\mu\D_\mu \bigr)
      \exp(x) \\
    &= \bigl( 1 - \D_\mu(1+\h\mu)^{-1}\h - \h(1+\mu\h)^{-1}\D_\mu \bigr) \exp(x)
    = \exp(x) .
  \end{align*}
  It follows that $\MC(i_\mu)\MC(p_\mu)=1$ on $\MC(L,h)$.  
\end{proof}

Applied to the simplicial contracting homotopy $s_\bullet$ on the simplicial
curved \Linf-algebra $\DR_\bullet\otimes L$, we obtain a natural identification
between the cofibration
\begin{equation*}
  \MC(i_\mu) : \MC(W_\bullet\otimes L) \to \MC_\bullet(L)
\end{equation*}
of fibrant simplicial sets, and the morphism
$\gamma_\bullet(L)=\MC(\DR_\bullet\ohat L,s_\bullet) \to \MC_\bullet(L)$. After this identification,
the cosection $\MC(p_\mu)$ of $\MC(i_\mu)$ is the holonomy map
$\rho:\MC_\bullet(L)\to\gamma_\bullet(L)$ of Theorem~\ref{main}.

It remains to discuss the functoriality of $\gamma_\bullet(L)$. Let
$f:L\to M$ be a fibration of curved \Linf-algebras.  From the explicit
formulas, together with the fact that $p_\mu\circ i_\mu$ is the identity on
$W_\bullet\otimes L$ and $W_\bullet\otimes M$ endowed with the curved \Linf-algebra structures
constructed above, we see that $f$ induces a strict morphism
$W_\bullet\otimes f$ from $W_\bullet\otimes L$ to
$W_\bullet\otimes M$.
\begin{proposition}
  The functor $\gamma_\bullet(L)$ is an exact functor from the strict category
  $\L$ of curved \Linf-algebras to the category $\S$ of fibrant
  simplicial sets.
\end{proposition}
\begin{proof}
  As in the proof of Proposition~\ref{exact}, we must show that for
  each $0<i\le n$, the morphism of curved \Linf-algebras
  \begin{equation*}
    W_n \otimes L \to \bigl( W(\Lambda^n_i)\otimes L \bigr) \times_{W(\Lambda^n_i)\otimes M} \bigl( W_n \otimes M
    \bigr)
  \end{equation*}
  is a trivial fibration, and for each $n\ge0$, the morphism of
  \Linf-algebras
  \begin{equation*}
    W_n\otimes L \to \bigl( W(\partial\triangle^n)\otimes L \bigr) \times_{W(\partial\triangle^n)\otimes M} \bigl( W_n\otimes M
    \bigr)
  \end{equation*}
  is a trivial fibration. But these are retracts in $\LL$ of the
  corresponding trivial fibrations \eqref{trivial:horn} and
  \eqref{trivial:boundary}, and the result follows.
\end{proof}

It is clear from the above discussion that the inclusion
$\gamma_\bullet(L)\hookrightarrow\MC_\bullet(L)$ and holonomy $\rho:\MC_\bullet(L)\to\gamma_\bullet(L)$ are natural
transformations of exact functors from $\L$ to $\S$.

\section*{$\ell$-groupoids}

A \textbf{thinness structure} $(X_\bullet,T_\bullet)$ on a simplicial set
$X_\bullet$ is a sequence of subsets $T_n\subset X_n$, $n>0$, of thin simplices
such that every degenerate simplex is thin.
\begin{definition}
  An \textbf{$\ell$-groupoid} is a simplicial set
  $(X_\bullet,T_\bullet)$ with thinness structure such that every horn has a
  unique thin filler, and every $n$-simplex is thin if $n>\ell$.

  A \textbf{strict $\ell$-groupoid} (or $T$-complex) is an
  $\ell$-groupoid such that the faces of the thin filler of a thin horn
  (a horn all of whose faces are thin) are thin.
\end{definition}

For $\ell<2$, every $\ell$-groupoid is strict. The nerve of a bigroupoid
\cite{Benabou} is a $2$-groupoid, but is a strict $2$-groupoid if and
only if the associator is trivial. For background to these
definitions, see Dakin~\cite{Dakin} and Ashley \cite{Ashley}, for the
strict case, and \cite{Linf} in general.

If $L$ is a curved \Linf-algebra concentrated in degrees
$[-\ell,\infty)$, then $\gamma_\bullet(L)$ is an $\ell$-groupoid: the thin
$n$-simplices are the Maurer--Cartan elements $x\in\DR_n\ohat L$ whose
component of top degree $n$ vanishes.

The following result was proved for nilpotent dg Lie algebras in the
special case $\ell=2$ in \cite{Linf}*{Proposition 5.8}. The proof in the
for general case is essentially the same.
\begin{proposition}
  If $L$ is a semiabelian curved \Linf-algebra and $L^k=0$ for
  $k<-\ell$, then $\gamma_\bullet(L)$ is a strict $\ell$-groupoid.
\end{proposition}
\begin{proof}
  A horn $y\in\Hom(\Lambda^n_i,\gamma_\bullet(L))$ is thin if and only if
  $y\in\DR(\Lambda^n_i)\ohat L^{\ge2-n}$. The extension $\sigma y$ of
  $y$ to $\triangle^n$ of Lemma~\ref{extension} satisfies
  $\sigma y\in\DR_n\ohat L^{\ge2-n}$.

  The thin filler $x\in\gamma_n(L)$ of $y$ is the limit
  $x = \lim_{k\to\infty} x_k$ where
  \begin{equation*}
    x_0 = \epsilon_n^iy + d(p_nh^i_n+s_n)\sigma y + \{(p_nh^i_n+s_n)\sigma y\}
  \end{equation*}
  and
  \begin{equation*}
    x_{k+1} = x_0 - \sum_{\ell=2}^\infty \frac{1}{\ell!} (p_nh_n^i+s_n) \{ x_k^{\otimes\ell}
    \} .
  \end{equation*}
  Since $L$ is semiabelian, $x_k\in\DR_n\otimes L^{\ge2-n}$ for all
  $k$. Hence $x\in\DR_n\ohat L^{\ge2-n}$, and $\partial_ix$ is thin.
\end{proof}

\section*{Acknowledgements}
\setstretch{1.15}

The idea for this paper was planted at the program on Higher
Categories and Categorification at MSRI/SLMath in Spring 2020. We are
grateful to Chris Kapulkin and the other participants in the seminar
on cubical sets. Ruggero Bandiera suggested some improvements to the
exposition of Theorem~\ref{Berglund}. We thank a referee for pointing
out the relevance of \cite{RV}. This paper is based upon work carried
out in the MSRI program ``Higher Categories and Categorification''
supported by the NSF under Grant No.\ DMS-1928930, and by Simons
Foundation Collaboration Grant 524522.


\end{document}